\documentclass[11pt, a4paper]{article}
\usepackage[english]{babel}
\usepackage[latin1]{inputenc}
\usepackage[T1]{fontenc}
\usepackage{amsmath}
\usepackage{amsfonts}
\usepackage{amssymb}
\usepackage{amsthm}
\usepackage[mathscr]{euscript}
\usepackage{graphicx}
\usepackage[affil-it]{authblk}
\usepackage{enumitem}
\usepackage[normalem]{ulem}
\usepackage{dsfont}
\usepackage{float}
\usepackage{times}
\usepackage{braket}

\usepackage[colorlinks,linkcolor=blue,citecolor=blue,urlcolor=blue]{hyperref}
\usepackage[capitalise]{cleveref}
\usepackage[left=4cm, right=4cm, bottom=2cm, top=2cm]{geometry}

\newtheorem{theorem}{Theorem}[section]
\newtheorem{lemma}[theorem]{Lemma}
\newtheorem{proposition}[theorem]{Proposition}
\newtheorem{corollary}[theorem]{Corollary}

\theoremstyle{definition}
\newtheorem{defn}[theorem]{Definition}
\theoremstyle{remark}
\newtheorem{remark}[theorem]{Remark}

\newtheorem{example}[theorem]{Example}
\numberwithin{equation}{section}

\crefname{example}{Example}{Examples}

\newcommand{\C}{\ensuremath{\mathcal{C}}}
\newcommand{\D}{\ensuremath{\mathcal{D}}}
\newcommand{\E}{\ensuremath{\mathcal{E}}}
\renewcommand{\H}{\ensuremath{\mathcal{H}}}

\newcommand{\N}{\ensuremath{\mathbb{N}}}

\newcommand{\R}{\ensuremath{\mathbb{R}}}

\newcommand{\1}{\ensuremath{\mathds{1}}}
\newcommand{\define}{\ensuremath{\triangleq}}

\renewcommand{\Re}{\mathrm{Re}}
\renewcommand{\Im}{\mathrm{Im}}
\renewcommand{\geq}{\geqslant}
\renewcommand{\leq}{\leqslant}
\DeclareMathOperator{\cov}{cov}
\DeclareMathOperator{\var}{var}
\DeclareMathOperator{\ent}{ent}
\DeclareMathOperator{\id}{id}

\newcommand{\vertiii}[1]{{\left\vert\kern-0.25ex\left\vert\kern-0.25ex\left\vert #1 \right\vert\kern-0.25ex\right\vert\kern-0.25ex\right\vert}}

\begin{document}
\title{Formulae for the Derivative of the Poincar\'e Constant of Gibbs Measures}
\author{Julian Sieber\footnote{Email: \href{mailto:j.sieber19@imperial.ac.uk}{j.sieber19@imperial.ac.uk}}}
\affil{Department of Mathematics, Imperial College London, Queen's Gate, London SW7 2AZ, United Kingdom}
\date{\today}

\maketitle
\begin{abstract}
	\noindent We establish formulae for the derivative of the Poincar\'e constant of Gibbs measures on both compact domains and all of $\R^d$. As an application, we show that if the (not necessarily convex) Hamiltonian is an increasing function, then the Poincar\'e constant is strictly decreasing in the inverse temperature, and vice versa. Applying this result to the $O(2)$ model allows us to give a sharpened upper bound on its Poincar\'e constant. We further show that this model exhibits a qualitatively different zero-temperature behavior of the Poincar\'e and Log-Sobolev constants.
	\vspace{0.2cm}
\end{abstract}

\section{Introduction and Main Results}
The study of functional inequalities such as the Poincar\'e and Logarithmic Sobolev inequality is of paramount interest in the stability theory of Markov processes. These functional inequalities can be used to determine the convergence rate of the law of an ergodic Markov process to the invariant measure. Beginning with Gross's seminal work linking Log-Sobolev inequalities to hypercontractivity of semigroups \cite{Gross1975}, functional inequalities have been successfully applied in a multitude of areas of mathematics; see the monographs \cite{Ane2000,Bakry1994, Bakry1983, Bakry2013, Guionnet2003, Ledoux2001, Martinelli1999, Royer2007, Saloff2002} for comprehensive overviews of the subject.

In this article we consider Gibbs probability measures $\mu_\beta$ at inverse temperature $\beta>0$ either on a bounded domain $U\subset\R^d$ or on all of $\R^d$. Recall that the Lebesgue density of $\mu_\beta$ is given by
\begin{equation}\label{eq:gibbs}
	d\mu_\beta(x)=\frac{1}{Z_\beta}e^{-\beta H(x)}\1_{U}(x)dx
\end{equation}
where the potential $H:U\to\R$ is such that $Z_\beta\define\int_U e^{-\beta H(x)}\,dx<\infty$.

We say that $\mu_\beta$ satisfies a \emph{Poincar\'e inequality} with constant $\alpha_\beta\geq 0$ if 
\begin{equation}\label{eq:poincare}
	\var_{\mu_\beta}(f)\define\int_U f^2\,d\mu_\beta-\left(\int_U f\,d\mu_\beta\right)^2\leq \alpha_\beta\int_U |\nabla f(x)|^2\,d\mu_\beta
\end{equation}
for all $f\in H^1(U,\mu_\beta)$. Here, the \emph{weighted Sobolev space} $H^1(U,\mu_\beta)$ is defined as the collection of weakly differentiable functions $f:U\to\R$ for which $f,\nabla f\in L^2(U,\mu_\beta)$. The norm
\begin{equation*}
	\|f\|_{H^1(U,\mu_\beta)}\define\left(\|f\|_{L^2(U,\mu_\beta)}^2+
	\|\nabla f\|_{L^2(U,\mu_\beta)}^2\right)^{1/2}
\end{equation*} 
renders it into a Hilbert space. The optimal (i.e. smallest) constant $\alpha_\beta$ in \eqref{eq:poincare} is called the \emph{Poincar\'e constant} of the measure $\mu_\beta$. Henceforth, $\alpha_\beta$ shall always denote this sharp constant. We also write $H^2(U)$ for the standard Sobolev space of twice weakly differentiable functions $f:U\to\R$ with $f,\nabla f,\nabla^2 f\in L^2(U)$. Finally, $W^{1,\infty}(U)$ shall denote the weakly differentiable functions $f:U\to\R$ with $f,\nabla f\in L^\infty(U)$. It is clear that $\C^1(\overline{U})\subset W^{1,\infty}(U)$ if $U$ is bounded.

The formulae of the derivative of the mapping $\beta\mapsto\alpha_\beta$ are summarized in the following \cref{thm:poincare_compact,thm:poincare_unbounded}. Together with \cref{thm:monotonicity} below, these are the main results of this article.
\begin{theorem}[Bounded domain]\label{thm:poincare_compact}
	Let $U\subset\R^d$ be a bounded domain. Assume further that $H\in W^{1,\infty}(U)$. If the smallest, non-zero eigenvalue of the infinitesimal generator $L_{\beta_0}$ is non-degenerate for some $\beta_0\geq 0$, then the Poincar\'e constant is analytic near $\beta_0$ and 
	\begin{equation}\label{eq:bounded_derivative}
		\partial_\beta\alpha_\beta\big|_{\beta=\beta_0}=-\alpha_{\beta_0}^2\int_U\varphi_{\beta_0}(x)\nabla{H}(x)\cdot\nabla\varphi_{\beta_0}(x)\,\mu_{\beta_0}(dx),
	\end{equation}
	where $\varphi_{\beta_0}$ is the normalized eigenfunction of the smallest, non-zero eigenvalue of the infinitesimal generator $L_{\beta_0}=\triangle-\beta_0 \nabla H\cdot\nabla$.
\end{theorem}
For simplicity, we only consider the whole space as unbounded domain:
\begin{theorem}[Unbounded domain]\label{thm:poincare_unbounded}
	Let $H:\R^d\to\R$ be twice weakly differentiable. Assume furthermore that 
	\begin{equation*}
		\sup_{x\in\R^d}\triangle H(x)<\infty\quad \text{and}\quad \lim_{|x|\to\infty}\big(|\nabla H(x)|^2-\triangle H(x)\big)=\infty.
	\end{equation*} 
	Then the Poincar\'e constant is analytic near any $\beta_0>0$ and the identity \eqref{eq:bounded_derivative} holds. Alternatively, we also have
	\begin{equation}\label{eq:unbounded_derivative}
		\partial_\beta\alpha_\beta\big|_{\beta=\beta_0}=-\frac{\alpha_{\beta_0}^2}{2}\int_{\R^d}\varphi_{\beta_0}^2(x)\big(\beta_0|\nabla{H}(x)|^2-\triangle H(x)\big)\,\mu_{\beta_0}(dx).
	\end{equation}
\end{theorem}
\begin{example}
	There are only a handful of Gibbs measures for which the Poincar\'e constant is explicitly known. The most prominent example is of course the Gaussian distribution $\mu_\beta(dx)\propto e^{-\frac{\beta x^2}{2}}$ on $\R$. It is very well known that this measure exhibits the sharp Poincar\'e constant $\alpha_\beta=\beta^{-1}$ and the inequality \eqref{eq:poincare} is saturated by constant multiples of the function $\varphi_\beta(x)=\sqrt{\beta}x$. The identities \eqref{eq:bounded_derivative} and \eqref{eq:unbounded_derivative} yield
	\begin{equation*}
		\partial_\beta\alpha_\beta=-\frac{1}{\beta}\int_\R x^2\,\mu_\beta(dx)=-\frac{1}{\beta^2},
	\end{equation*}
	as expected. By tensorization, the case of a $d$-dimensional standard Gaussian distribution can be verified similarly.
\end{example}

As an application of the formula \eqref{eq:bounded_derivative} we would like to investigate the monotonicity of the mapping $\beta\mapsto\alpha_\beta$. If the Hamiltonian $H$ is $\kappa$-semi-convex for some $\kappa>0$, that is, $x\mapsto H(x)-\frac\kappa2 \|x\|^2$ convex (or equivalently $\nabla^2 H\geq\kappa$ if $H\in\C^2$), then the Bakry-\'Emery criterion \cite{Bakry1983} tells us that
\begin{equation}\label{eq:log_conc}
	\alpha_\beta\leq \frac{1}{\beta\kappa}.
\end{equation}
This provides a strong indication that the Poincar\'e constant is strictly decreasing in this setting.

The main difficulty we face in establishing this is the fact that the relationship between the potential $H$ and the eigenfunction $\varphi_\beta$ is unclear in general. If we however work in one dimension and impose certain symmetry assumptions, we can deduce a connection between $H$ and the monotonicity of the Poincar\'e constant. This is our final main result. To formulate it, let us recall that we say $H:I\to\R$, $I\subset\R$ an interval, is \emph{piecewise $\C^1$} if it is continuous and there is a finite subset of kinks $\mathcal{K}\subset I$ such that $H\restriction_{I\setminus\mathcal{K}}$ is continuously differentiable. At the exceptional points $\mathcal{K}$, we require the left- and right-sided derivative to exist.

\begin{theorem}\label{thm:monotonicity}
	Assume that $U=(-a,a)$ for some $a>0$ and $H:U\to\R$ is non-constant, even, and piecewise $\C^1$ with kinks $\mathcal{K}$. Then the sharp Poincar\'e constant $\alpha_\beta$ is a strictly 
	\begin{equation*}
		\begin{cases}
			\text{decreasing,}& \text{if }H^\prime(x)\geq 0\text{ for all }x\in(0,a)\cap\mathcal{K}^c,\\
			\text{increasing,}& \text{if }H^\prime(x)\leq 0\text{ for all }x\in(0,a)\cap\mathcal{K}^c,
		\end{cases}
	\end{equation*}
	function of $\beta\geq 0$. 
\end{theorem}
The following corollary makes the decay hinted by \eqref{eq:log_conc} rigorous.
\begin{corollary}\label{cor:log_concave}
	Let $U=(-a,a)$ and $H\in\C^1(U)$ be even and $\kappa$-semi-convex for some $\kappa>0$. Then the Poincar\'e constant $\alpha_\beta$ is strictly decreasing in $\beta$.
\end{corollary}
\begin{proof}
	Since $H$ is even and continuously differentiable, we must have $H^\prime(0)=0$. By $\kappa$-semi-convexity, it follows that $H^\prime(x)\geq\kappa x$ for all $x\in(0,a)$.
\end{proof}
There are, however, a variety of Gibbs measures of practical importance that are not comprised by \cref{cor:log_concave}, but which do fall in the regime of \cref{thm:monotonicity} (see \cref{fig:theorem} for the illustration of such a potential). 

\begin{figure}[H]
	\centering
    \includegraphics[width=0.7\textwidth]{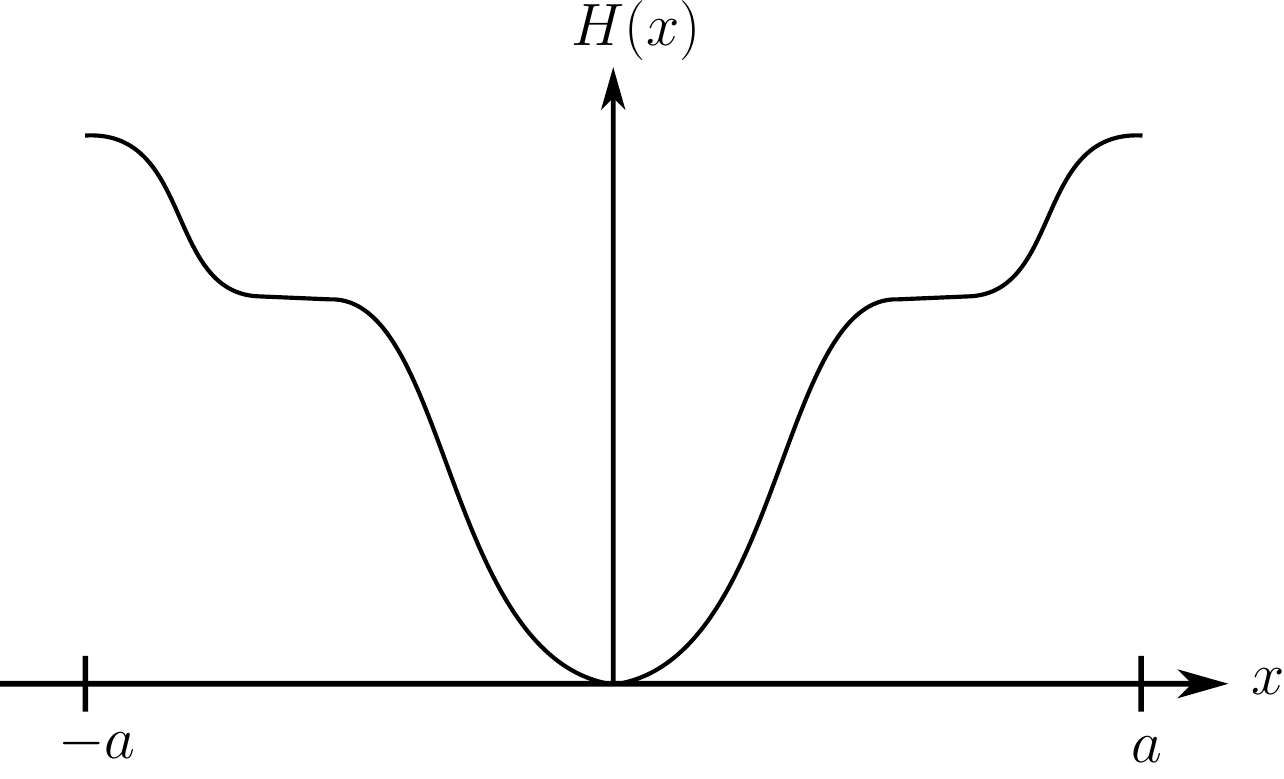}
	\caption{A prototypical function satisfying the assumptions of \cref{thm:monotonicity}. Notice that the Hamiltonian $H$ may have saddle points or even flat passages. Furthermore, $H$ need not be convex.}
	\label{fig:theorem}
\end{figure}

The article is structured as follows: In \cref{sec:derivative} we recall some results from the perturbation theory of self-adjoint operators. \Cref{sec:monotonicity} establishes symmetry properties of the eigenfunction $\varphi_\beta$ which are then used to deduce \cref{thm:monotonicity}. There, we also provide an example illustrating that the statement fails for the Log-Sobolev constant. One of the prime applications of \cref{thm:monotonicity} might be to prove bounds on the Poincar\'e constant which are uniform in $\beta$. We illustrate this in \cref{sec:o_n} on the example of the $O(2)$ model. We conclude this article in \cref{sec:saturation} with another noteworthy property of the $O(2)$-model: The Poincar\'e constant vanishes as $\beta\to\infty$, whereas the Log-Sobolev constant saturates at a strictly positive value.

\paragraph{Acknowledgments.} I would like to thank Sebastian Andres and Roland Bauerschmidt for their encouragement and many helpful discussions. I am very grateful to two anonymous referees for their careful reading of the paper. Their comments considerably improved the presentation. Financial support by the EPSRC under grant no. EP/S023925/1 is gratefully acknowledged.

\section{Derivative of the Poincar\'e Constant}\label{sec:derivative}

Without further notice, in the sequel we always require that the domain of any unbounded operator (or a quadratic form) is dense. Remember that the resolvent set $\rho(L)$ of an operator $\big(L,\D(L)\big)$ consists of the values $\rho\in\mathbb{C}$ such that $\rho-L$ is bounded invertible. The spectrum of $L$ is defined as the complement of the resolvent set: $\sigma(L)\define\mathbb{C}\setminus\rho(L)$. The operator $L$ has compact resolvent if for some (and hence all) $\rho\in\rho(L)$, $(\rho-L)^{-1}$ defines a compact operator. It is standard that, if in addition $\big(L,\D(L)\big)$ is self-adjoint, this is equivalent to a purely discrete spectrum which may only accumulate at $\pm\infty$. In fact, sufficiency is known to even hold for closed operators on Banach spaces and necessity is an easy consequence of the spectral theorem.

\subsection{The Poincar\'e Constant as Eigenvalue Problem}\label{sec:eigenvalue}

Let $U\subset\R^d$. Recall that the Poincar\'e constant is given by
\begin{equation*}
	\alpha_\beta=\sup_{\substack{f\in H^1(U,\mu_\beta)\\f\perp 1}}\frac{\E_\beta(f)}{\|f\|^2_{\beta}},
\end{equation*}
where $\E_\beta(f)\define\int_U |\nabla f|^2\,d\mu_\beta$, $f\perp 1$ indicates that $f$ is centered with respect to $\mu_\beta$, and $\|\cdot\|_{L^2(U,\mu_\beta)}$ denotes the norm on $L^2(U,\mu_\beta)$. We also write $\braket{\cdot,\cdot}_\beta$ for the associated inner product. Up to subtleties regarding the form domain of $\E_\beta$, which we shall address below, the Poincar\'e constant is therefore nothing but the minimizer of the Rayleigh quotient, characterizing the second smallest eigenvalue of a suitable self-adjoint operator. In fact, if we define
\begin{equation*}
	\mathcal{A}\define\begin{cases}
	\left\{f\in\C^\infty(\overline{U}):\,\frac{\partial f}{\partial n}=0\text{ on }\partial U\right\}, & U\subset\R^d\text{ bounded},\\
	\C_c^\infty(\R^d), & U=\R^d,
	\end{cases}
\end{equation*}
where $\frac{\partial f}{\partial n}$ denotes the normal derivative of $f$, then it is easy to check that 
\begin{equation*}
	L_\beta\define\triangle-\beta\nabla H\cdot\nabla,\qquad \D(L_\beta)=\mathcal{A},
\end{equation*}
is a non-positive, densely defined, symmetric operator with $\E_\beta(f)=-\braket{f,L_\beta f}_\beta$ its Dirichlet form. In a slight abuse of notation, we denote its self-adjoint Friedrichs extension of $L_\beta$ by the same symbol. Remember that the domain of the associated Dirichlet form is precisely given by $\D(\E_\beta)=\overline{\mathcal{A}}^{\|\cdot\|_{\E_\beta}}$ where the closure is taken with respect to the form norm $\|f\|_{\E_\beta}=\|f\|_{H^1(U,\mu_\beta)}$. 

In case $U=\R^d$, the operator $\big(L_\beta,\D(L_\beta)\big)$ is of course the infinitesimal generator of the overdamped Langevin (or stochastic gradient) diffusion
\begin{equation*}
	dX_t^\beta=-\beta\nabla H(X_t^\beta)\,dt+\sqrt{2}\,dW_t
\end{equation*}
with $(W_t)_{t\geq 0}$ a $d$-dimensional Wiener process. If $U\subset\R^d$ is bounded, the question of the associated Markov process is much more delicate because of boundary effects encoded in the construction of $\big(L_\beta,\D(L_\beta)\big)$. The natural candidate diffusion is given by
\begin{equation}\label{eq:normal_reflection}
	dX_t^\beta=-\beta\nabla H(X_t^\beta)\,dt+\sqrt{2}\,dW_t-\ell_t^\beta,
\end{equation}
where the bounded variation process $\ell^\beta$ is implicitly defined by
\begin{equation*}
	\ell_t^\beta=\int_0^t n(X_s^\beta)\,d|\ell^\beta|_s,\qquad |\ell^\beta|_t=\int_0^t \1_{\partial U}(X_s^\beta)\,d|\ell^\beta|_s
\end{equation*}
with the total variation process $|\ell^\beta|$. Well-posedness of \eqref{eq:normal_reflection} (in the weak sense) is related to the unique solvability of the celebrated Skohorod problem \cite{Skorokhod1961,Skorokhod1962}. Under additional regularity assumptions on both the potential $H$ and the domain $U$, this can indeed be established \cite{Stroock1971,Lions1984}. The existence of the associated Markov process is however not pertinent to questions investigated in this article.

It turns out the operator $L_\beta$ is unitarily equivalent to a Schr\"odinger operator $\tilde{L}_\beta$ on the unweighted space $L^2(U)$. To see this, we define the \emph{inverse ground state transformation} (or \emph{$h$-transform}) $U_\beta:L^2(U)\to L^2(U,\mu_\beta)$, $U_\beta f(x)=\sqrt{Z_\beta}e^{\beta H(x)/2}f(x)$. It is easy to check that $U_\beta$ is unitary and
\begin{equation}\label{eq:ground_state_trafo}
	\tilde{L}_\beta f(x)\define U^*_\beta L_\beta U_\beta f(x)=\triangle f(x)-\frac12\left(\frac{\beta^2}{2}|\nabla H(x)|^2-\beta\triangle H(x)\right)f(x).
\end{equation}

The following result summarizes the connection of the Poincar\'e constant with the spectrum of the operator $L_\beta$, taking into account the due domain considerations. Without further notice, we shall tacitly assume the conditions of \cref{thm:poincare_compact,thm:poincare_unbounded}, respectively, on the regularity of the domain $U$ and the potential $H$.
\begin{proposition}\label{prop:neumann_advanced}
	Let $\big(L_\beta,\D(L_\beta)\big)$ be the Friedrichs extension of the operator $L_\beta=\triangle-\beta\nabla H\cdot\nabla$ with domain $\mathcal{A}$ as above. This operator has a discrete spectrum for any $\beta\geq 0$ (for any $\beta>0$ if $U=\R^d$) and the Poincar\'e constant is given by $\alpha_\beta=-(\lambda_1^\beta)^{-1}$ where $\lambda_1^\beta=\sup\{\lambda\in\sigma(L_\beta):\,\lambda<0\}$. Moreover, the inequality \eqref{eq:poincare} is saturated if and only if $f\in H^1(U,\mu_\beta)$ is in the associated eigenspace. If $U=\R^d$, the eigenvalue $\lambda_1^\beta$ is non-degenerate.
\end{proposition}
\begin{proof}
	Consider first the setting of \cref{thm:poincare_compact}. Because of the boundedness of the potential $H$, it is clear that the spaces $H^{1}(U,\mu_\beta)$ and $H^{1}(U)$ coincide for each $\beta\geq 0$. Moreover, the respective norms are equivalent. It is known that, under mild regularity assumptions on $\partial U$ (which certainly comprise a piecewise smooth boundary), $\D(\E_\beta)=\overline{A}^{\|\cdot\|_{H^1(U)}}=H^1(U)=H^1(U,\mu_\beta)$, see e.g. \cite[p. 263]{Reed1978}. Consequently, $L_\beta$ is a second-order elliptic differential operator with Neumann boundary conditions. Discreteness of the spectrum is thus ensured by standard results. The remaining claims are immediate from a well-known variational characterization of the eigenvalues and continuity of both sides in the inequality \eqref{eq:poincare} with respect to the norm $\|\cdot\|_{H^1(U,\mu_\beta)}$.

	If $U=\R^d$, we employ \cite[Theorem XIII.47]{Reed1978} to deduce that $\tilde{L}_\beta$ (see \eqref{eq:ground_state_trafo}) has discrete spectrum and the largest non-zero eigenvalue $\lambda_1^\beta$ is non-degenerate. The remaining claims are then immediate. 
\end{proof}

In view of \cref{prop:neumann_advanced} we are led to compute the derivative of the first non-vanishing eigenvalue of $L_\beta$. If we forget about the technical details for a moment, then \cref{thm:poincare_compact,thm:poincare_unbounded} can be proven by the following formal computation: Let $\varphi_\beta$ be a normalized eigenvector of $L_\beta$ with eigenvalue $\lambda_\beta$. Then
\begin{align*}
	\partial_\beta\lambda_\beta&=\partial_\beta\int_U \varphi_\beta(x)L_\beta \varphi_\beta(x)\,\mu_\beta(dx)\\
	&=2\int_U \big(\partial_\beta \varphi_\beta\big)(x)L_\beta \varphi_\beta(x)\,\mu_\beta(dx)+\int_U \varphi_\beta(x)L_\beta \varphi_\beta(x)\big(\partial_\beta \varrho_\beta\big)(x)\,dx\\
	&\phantom{=}-\int_U \varphi_\beta(x) \nabla H(x)\cdot\nabla \varphi_\beta(x)\,\mu_\beta(dx),
\end{align*}
where $\varrho_\beta(x)\define Z_\beta^{-1}e^{-\beta H}(x)$. Since
\begin{align*}
	2\int_U &\big(\partial_\beta \varphi_\beta\big)(x)L_\beta \varphi_\beta(x)\,\mu_\beta(dx)+\int_U \varphi_\beta(x)L_\beta \varphi_\beta(x)\big(\partial_\beta \varrho_\beta\big)(x)\,dx\\
	&=\lambda_\beta\partial_\beta\|\varphi_\beta\|_{L^2(U,\mu_\beta)}^2=0,
\end{align*}
the identity \eqref{eq:bounded_derivative} follows. It is clear that this computation is problematic in a number of ways. The next section is therefore devoted to provide a rigorous backing of these manipulations.

\subsection{Perturbation Theory of Infinitesimal Generators}

Consider a family operators $(L_\beta,\D(L_\beta))_{\beta\in D}$ on a Hilbert space $\H$ indexed by a parameter $\beta$ with values in a non-empty domain $D\subset\mathbb{C}$. By this we mean that $D$ is open and connected. Of course, eventually we will be interested in taking $\beta\in D_\R\define D\cap\R$, but for now it is more convenient to allow complex $\beta$.

Let us begin with a definition:
\begin{defn}\label{def:holomorphic}
Let $D\subset\mathbb{C}$ be a non-empty domain. We say that $(L_\beta)_{\beta\in D}$ defines a \emph{holomorphic family} if the following hold:
\begin{enumerate}
	\item For each $\beta\in D$, $L_\beta$ is closed and has a non-empty resolvent set,
	\item for each $\beta_0\in D$, there is a $\rho_0\in\rho(L_{\beta_0})$ such that $\rho_0\in\rho(L_{\beta})$ for $\beta$ near $\beta_0$ and $\beta\mapsto(\rho_0-L_\beta)^{-1}$ defines a holomorphic function at $\beta_0$.
\end{enumerate}
If we can choose $D=\mathbb{C}$, we call $(L_\beta)_{\beta\in\mathbb{C}}$ an \emph{entire family}.
\end{defn}
Note that if $(L_\beta)_{\beta\in D}$ is a holomorphic family on $\H$ and $U:\H\to\H^\prime$ is unitary, then $(UL_\beta U^*)_{\beta\in D}$ is a holomorphic family on $\H^\prime$, see \cite[p. 20]{Reed1978}. In many applications, it turns out to be notoriously difficult to check \cref{def:holomorphic} directly. Instead, it is often more convenient to work with quadratic forms. We recall that a quadratic form $\big(\E,\D(\E)\big)$ is called \emph{sectorial} if there are a base point $y\in\mathbb{R}$ and a semi-angle $\theta\in[0,\frac{\pi}{2})$ such that the \emph{numerical range} of $\E$ is contained in the sector $S_{y,\theta}$:
\begin{equation*}
	\{\E(f):\,f\in\D(\E)\}\subset S_{y,\theta}\define\{w\in\mathbb{C}:\,|\arg(w-y)|\leq\theta\}.
\end{equation*}
Observe that this is equivalent to
\begin{equation}\label{eq:sector}
	\Re\big(\E(f)\big)\geq y\|f\|^2\;\text{and}\;\Big|\Im\big(\E(f)\big)\Big|\leq\tan(\theta)\Big(\Re\big(\E(f)\big)-y\|f\|^2\Big)\quad\forall f\in\D(\E).
\end{equation}
We have the following well-known generalization of the Friedrichs extension: For any closed, sectorial form $\big(\E,\D(\E)\big)$, there is a unique closed operator $\big(L,\D(L)\big)$ with $\E(f)=-\braket{f,Lf}$ for all $f\in\D(L)\subset\D(\E)$ and $\D(L)$ is a form core of $\E$.
\begin{defn}\label{def:type_b}
	Let $D\subset\mathbb{C}$ be a non-empty domain. We say that $(\E_\beta)_{\beta\in D}$ defines a \emph{holomorphic family of type (a)} if the following hold:
	\begin{enumerate}
		\item For each $\beta\in D$, $\big(\E_\beta,\D(\E_\beta)\big)$ is a closed, sectorial form,
		\item $\D(\E_\beta)=\D(\E)$ is independent of $\beta\in D$,
		\item for each $f\in\D(\E)$, $\beta\mapsto\E_\beta(f)$ is holomorphic on $D$.
	\end{enumerate}
	The associated operators $(L_\beta)_{\beta\in D}$ are called a \emph{holomorphic family of type (B)}.
\end{defn}
One can check (see \cite[Theorem VII.4.2]{Kato1995}) that a holomorphic family of type (B) is holomorphic in the sense of \cref{def:holomorphic}. The following result is proven in \cite[Theorem VII.4.8]{Kato1995}:
\begin{lemma}\label{lem:criterion_type_b}
	Suppose that $\big(\E_1,\D(\E_1)\big)$ is a closed, non-negative form. Let $\big(\E_2,\D(\E_2)\big)$ be a quadratic form with $\D(\E_1)\subset\D(\E_2)$. If there are $a,b>0$ such that
	\begin{equation*}
		\E_{2}(f)\leq a\E_1(f)+b\|f\|^2
	\end{equation*}
	for all $f\in\D(\E_1)$, then there is a domain $0\in D\subset\mathbb{C}$ such that $\E_\beta\define\E_1+\beta\E_2$ is closable and its closure defines a holomorphic family of type (a). If $a>0$ can be chosen arbitrarily small, then $(\E_\beta)_{\beta\in\mathbb{C}}$ is actually an entire family of type (a).
\end{lemma}

It is well known that, if $(L_\beta)_{\beta\in D}$ is a holomorphic family of type (B) and has compact resolvent for some $\beta_0\in D$, then the same holds for any $\beta\in D$, see \cite[Theorem VII.2.4]{Kato1995}. In the setting of \cref{thm:poincare_unbounded} we can therefore not expect that the family of infinitesimal generators is analytic near $\beta_0=0$. In fact, the Ornstein-Uhlenbeck operator $L_\beta=\triangle-\frac12\beta |x|^2$ is well known to be of compact resolvent for any $\beta>0$ and, as we shall see below, $(L_\beta)_{\beta\in D}$ is a holomorphic family of type (B) on a neighborhood of $(0,\infty)$. However, it is very well known that the spectrum of the Laplacian is continuous, whence it cannot be of compact resolvent.

We say that two inner products on a linear space are equivalent if the induced norms are equivalent. Our main abstract result in this section is as follows:
\begin{proposition}\label{prop:derivative}
	Let $D\subset\mathbb{C}$ be a domain and let $\H$ be a Hilbert space equipped with a family of pairwise equivalent inner products $\{\braket{\cdot,\cdot}_\beta\}_{\beta\in D_\R}$. Let $(L,\D(L))$, $(V_1,\D(V_1))$, and $(V_2,\D(V_2))$ be densely defined operators on $\H$. Suppose that $L_\beta\define L+\beta V_1+\beta^2 V_2$, $\beta\in D$, defines a holomorphic family and, for each $\beta\in D$, $L_\beta$ has compact resolvent. If $L_\beta$ is self-adjoint with respect to $\braket{\cdot,\cdot}_\beta$ for each $\beta\in D_\R$, then the following are true:
	\begin{enumerate}
	 	\item\label{it:all_eigenvalues} The eigenvalues $\{\lambda_n(\beta)\}_{n\in\N}$ (counted with multiplicity) are analytic near each $\beta_0\in D_\R$.
	 	\item\label{it:single_eigenvalues}  If $\lambda_n(\beta_0)$ is non-degenerate, then we have the formula
		\begin{equation}\label{eq:feynman_hellmann}
			\frac{d}{d\beta}\lambda_n(\beta)\Big|_{\beta=\beta_0}=\braket{\varphi_n(\beta_0),(V_1+2\beta_0 V_2) \varphi_n(\beta_0)}_{\beta_0},
		\end{equation}
	where $\varphi_n(\beta)$ is the associated normalized eigenfunction.
	 \end{enumerate}
\end{proposition}

\begin{remark}\label{rem:crossings}
	We insist that under the assumptions of \cref{prop:derivative} crossings of the eigenvalues may occur. In other words, if we assume that $\lambda_1(0)\leq\lambda_2(0)\leq\cdots$, this may not hold for arbitrary $\beta\in D_\R$! The Poincar\'e constant is in general not differentiable at a crossing of the smallest non-zero eigenvalue.
\end{remark}

\begin{proof}[Proof of \cref{prop:derivative}]
	Without any loss of generality, we may assume that $0\in D$ and $\beta_0=0$. We also drop the index of the eigenvalue and the eigenvector for brevity throughout the proof. Since $L$ has compact resolvent, we can find an $\varepsilon>0$ such that $\sigma(L)\cap B_\varepsilon\big(\lambda(0)\big)=\{\lambda(0)\}$, albeit with possible degeneracy. The projection-valued function 
	\begin{equation}\label{eq:projection}
		P_\beta=\frac{1}{2\pi i}\oint_{\partial B_{\varepsilon}(\lambda(0))}\big(z-L_\beta\big)^{-1}\,dz
	\end{equation}
	is thus holomorphic near $0$ and projects on the (potentially multi-dimensional) eigen\-space of $\lambda(\beta)$. In other words, we have that $\sigma\big(P_\beta L_\beta P_\beta\big)=\sigma(L_\beta)\cap B_\varepsilon\big(\lambda(0)\big)$. Moreover, for real $\beta$, $P_\beta$ is symmetric with respect to $\braket{\cdot,\cdot}_\beta$. A classical result of Kato \cite{Kato1950} (see also \cite[p. 386]{Kato1995}) states that there is a holomorphic family of bounded, invertible operators $\{U_\beta\}$ for $\beta$ near $0$ such that $P_\beta=U_\beta P_0 U_\beta^{-1}$ and, for real $\beta$, $U_\beta$ is unitary with respect to $\braket{\cdot,\cdot}_\beta$. We now set $\bar{L}_\beta\define P_0U_\beta^{-1}L_\beta U_\beta P_0$. This defines a finite-dimensional holomorphic family, which is symmetric for real $\beta$. Analyticity of the eigenvalues therefore follows by a standard argument, see e.g. \cite[Theorem II.6.1]{Kato1995}. This concludes the proof of \ref{it:all_eigenvalues}.

	For point \ref{it:single_eigenvalues} we compute the series expansion for non-degenerate $\lambda(0)$. To this end, we write $\bar{L}$, $\bar{V}_1$, and $\bar{V}_2$ for the restrictions of these operators to the range of $P_0$. Let $\varphi(0)$ be the normalized eigenvector with eigenvalue $\lambda(0)$. For $\beta$ near $0$, we see that $P(\beta)\varphi(0)\neq 0$ (since $P(\beta)\varphi(0)\to \varphi(0)$ as $\beta\to 0$). Consequently,
	\begin{align*}
		\lambda(\beta)&=\frac{\braket{\varphi(0),\bar{L}_\beta P(\beta)\varphi(0)}_0}{\braket{\varphi(0),P(\beta)\varphi(0)}_0}\\
		&=\lambda(0)+\beta\frac{\braket{\varphi(0),\bar{V}_1 P(\beta)\varphi(0)}_0}{\braket{\varphi(0),P(\beta)\varphi(0)}_0}+\beta^2\frac{\braket{\varphi(0),\bar{V}_2 P(\beta)\varphi(0)}_0}{\braket{\varphi(0),P(\beta)\varphi(0)}_0}.
	\end{align*}
	 The integrand in \eqref{eq:projection} is holomorphic on the contour of integration and, iterating the second resolvent identity, we can expand
	\begin{equation*}
		\big(z-(\bar{L}+\beta \bar{V}_1+\beta^2 \bar{V}_2)\big)^{-1}=\big(z-\bar{L}\big)^{-1}+\beta\big(z-\bar{L})^{-1}\bar{V}_1\big(z-\bar{L})^{-1}+\mathcal{O}(\beta^2),
	\end{equation*}
	whence
	\begin{align*}
		\lambda(\beta)&=\lambda(0)+\beta\frac{\braket{\varphi(0),\bar{V}_1 \varphi(0)}_0+\mathcal{O}(\beta)}{1+\mathcal{O}(\beta)}+\mathcal{O}(\beta^2)\\
		&=\lambda(0)+\beta\braket{\varphi(0),V_1 \varphi(0)}_0+\mathcal{O}(\beta^2).
	\end{align*}
	Here, we used that 
	\begin{equation*}
		\frac{1}{2\pi i}\oint_{\partial B_\varepsilon(\lambda(0))}(z-L)^{-1}\varphi(0)\,dz=\frac{1}{2\pi i}\oint_{\partial B_\varepsilon(\lambda(0))}\big(z-\lambda(0)\big)^{-1}\varphi(0)\,dz=\varphi(0).
	\end{equation*}
\end{proof}

\section{Proofs of the Main Results}\label{sec:monotonicity}

Given a self-adjoint operator $\big(L,\D(L)\big)$ let $\big(\E_L,\D(\E_L)\big)$ denote the unique, closed quadratic form such that $\E_L(f)\define-\braket{f,Lf}$ for all $f\in\D(L)$. Combining the results of the previous section, it is now easy to deduce \cref{thm:poincare_compact}:
\begin{proof}[Proof of \cref{thm:poincare_compact}]
	We choose $\H=L^2(U)$ and observe that, since $H$ is bounded,
	\begin{equation*}
		\braket{f,g}_\beta=\frac{1}{Z_\beta}\int_U f(x)g(x)e^{-\beta H(x)}\,dx
	\end{equation*}
	defines a family of pairwise equivalent inner products.	Define $V\define\-\nabla H\cdot\nabla$. Since $\nabla H$ is bounded, it follows that
	\begin{equation*}
		\E_{V}(f)\leq\|\nabla H\|_\infty\left(\frac{\varepsilon}{2}\E_\triangle(f)+\frac{1}{2\varepsilon}\|f\|^2\right)
	\end{equation*}
	for any $\varepsilon>0$. Owing to \cref{lem:criterion_type_b}, $L_\beta$ defines an entire family. Standard results in the theory of partial differential equations tell us that $L_\beta$ has compact resolvent for any $\beta\geq 0$, see e.g. \cite[Theorem 3.5.1]{Evans1998}. The theorem is now an immediate consequence of \cref{prop:derivative}.
\end{proof}

For unbounded domains the situation becomes of course much more delicate. While on bounded domains $L_\beta$ actually has compact resolvent for all $\beta\in\mathbb{C}$, this is now no longer true. However, we shall see below that, under the assumptions of \cref{thm:poincare_unbounded}, the spectrum is still discrete for $\beta>0$. Before that, let us check that $(\tilde{L}_\beta)_{\beta\in D}$ constitutes a holomorphic family of type (B) for 
\begin{equation*}
	D\define\left\{z\in\mathbb{C}\setminus\{0\}:\,|\arg z|<\frac{\pi}{6}\right\}.
\end{equation*}
To this end, we observe that the operator \eqref{eq:ground_state_trafo} naturally decomposes as
\begin{align*}
	\tilde{L}_\beta f(x)&=\triangle f(x)-\frac12\left(\frac{\beta^2}{2}|\nabla H(x)|^2-\beta\triangle H(x)\right)f(x)\\
	&\define \triangle f(x) - \left(\beta^2 V_1(x)-\beta V_2^+(x)+\beta V_2^-(x)\right)f(x),
\end{align*}
where 
\begin{equation*}
	V_1(x)\define\frac14|\nabla H(x)|^2,\quad
	V_2^+(x)\define\frac12\big(0\vee \triangle H(x)\big),\quad V_2^-(x)\define-\frac12\big(0\wedge \triangle H(x)\big).
\end{equation*}
Without further notice, we shall assume the conditions of \cref{thm:poincare_unbounded} in the following two lemmas:

\begin{lemma}\label{lem:sector}
	Let $\E_{\triangle+\beta V^+_2}\define \E_{\triangle}+\beta \E_{V^+_2}$ with domain $\D(\E_{\triangle+\beta V^+_2})\define\D(\E_\triangle)$. Then, for each $\beta\in D$ and $\theta\in(0,\frac{\pi}{2})$, there is a $y<0$ such that the numerical range of $\E_{\triangle+\beta V^+_2}$ is contained in $S_{y,\theta}$.
\end{lemma}
\begin{proof}
	It is clear that $|V_2^+(x)|\leq\frac12 \big(\sup_{x\in\R^d}\triangle H(x)\vee 0\big)\define a$. Therefore,
	\begin{equation*}
		\Big|\Im\big(\E_{\triangle+\beta V^+_2}(f)\big)\Big|=\big|\Im(\beta)\big|\big|\E_{\beta V^+_2}(f)\big|\leq a\big|\Im(\beta)\big| \|f\|^2
	\end{equation*}
	and recalling \eqref{eq:sector} the lemma follows at once.
\end{proof}

\begin{lemma}\label{lem:sectorial}
	For each $\beta\in D$, the quadratic form $\tilde{\E}_\beta(f)\define-\braket{f,\tilde{L}_\beta f}$, $\D(\tilde{\E}_\beta)\define\D(\E_{\triangle})\cap\D(\E_{-V_1})\cap\D(\E_{-V_2^-})$, is densely defined, closed, and sectorial.
\end{lemma}
\begin{proof}
	Recall that, for any function $V:\R^d\to\R$, the multiplication operator $(Vf)(x)\define V(x)f(x)$ is self-adjoint on the domain $\D(L)\define\big\{f\in L^2(\R^d):\, Vf\in L^2(\R^d)\big\}$. Therefore, we certainly have that $\C_c^\infty(\R^d)\subset\D(\tilde{\E}_\beta)$ and hence the latter is dense. Fix $\theta\in\big(0,\frac{\pi}{2}-3\arg(\beta)\big)$ and let $y<0$ be the associated base point furnished by \cref{lem:sector}. The forms $\beta^2 \E_{V_1}$ and $\beta\E_{V_2^-}$ are sectorial with $S_{y,2\arg(\beta)}$ and $S_{y,\arg(\beta)}$ respectively. Since $S_{y,\theta}+S_{y,\tilde{\theta}}\subset S_{y,\theta+\tilde{\theta}}$, it follows that $\tilde{\E}_\beta=\E_{\triangle+\beta V_2^+}-\beta^2 \E_{V_1}-\beta \E_{V_2^-}$ is sectorial.

	It remains to show that $\tilde{\E}_\beta$ is closed. To this end, we recall that---without loss of generality---we may consider the form norm
	\begin{equation*}
		\vertiii{f}_{\tilde{\E}_\beta}\define\Big|\Braket{f,\big(\triangle-\beta^2 V_1-\beta V_2^-+\beta V_2^++\Re(\beta)a+1\big)f}_{L^2(\R^d)}\Big|,
	\end{equation*}
	where, as before, $a>0$ is chosen such that $|V_2^+|\leq a$. It is now easy to see that $\vertiii{f_n}_{\tilde{\E}_\beta}\to 0$ implies $\vertiii{f_n}_{\E_\sharp}\define\Braket{f_n,\big(-\sharp+1\big) f_n}\to 0$ for each $\sharp\in\{\triangle, -V_1,-V_2^-\}$. Hence, $\big(\tilde{\E}_\beta,\D(\tilde{\E}_\beta)\big)$ is closed.
	\end{proof}

We can now prove our second main result:
\begin{proof}[Proof of \cref{thm:poincare_unbounded}]
	Owing to \cite[Theorem XIII.47]{Reed1978}, for each $\beta>0$, the spectrum of $\tilde{L}_\beta$ (and hence $L_\beta$) is discrete and the ground state is non-degenerate. Moreover, it follows from \cref{lem:sectorial} that $(\tilde{L}_\beta)_{\beta\in\mathcal{S}}$ constitutes a holomorphic family of type (B). The theorem is therefore an immediate consequence of \cref{prop:derivative}. 
\end{proof}

It remains to establish \cref{thm:monotonicity}. To this end, we of course show that the right-hand side of \eqref{eq:bounded_derivative} is strictly positive (negative). This will be based on the following essentially well-known result:
\begin{proposition}\label{prop:sl_properties}
Let $\beta\geq 0$ and $U=(-a,a)$, $a>0$. Moreover assume that $H:U\to\R$ is piecewise $\C^1$. Consider the Sturm-Liouville problem
\begin{align*}
	f^{\prime\prime}(x)-\beta H^\prime(x)f^\prime(x)&=\lambda^\beta f(x),\\
	f^\prime(-a)=f^\prime(a)&=0,
\end{align*}
on $U$. Denote its $k^\text{th}$ eigenvalue and normalized eigenfunction by $\lambda_k^\beta$ and $\varphi_k^\beta$, respectively. Then the following hold true:
\begin{enumerate}
	\item\label{it:properties_1} The mapping $x\mapsto\varphi_k^\beta(x)$ is differentiable and
	\begin{equation}\label{eq:eigenfunction_derivative}
	(\varphi_k^\beta)^\prime(x)=\frac{\lambda_k^\beta}{e^{-\beta H(x)}}\int_{x}^{a}\varphi_k^\beta(t)e^{-\beta H(t)}\,dt
	\end{equation}
	for all $k\in\N_0$ and Lebesgue-a.e. $x\in U$.
	\item\label{it:properties_2} The eigenfunction $\varphi_1^\beta$ is strictly monotone and odd.
\end{enumerate}
\end{proposition}
The representation \eqref{eq:eigenfunction_derivative} follows immediately from the fact that $H^1(\mu_\beta,U)$ consists precisely of all absolutely continuous functions. Item \ref{it:properties_2} is then an easy consequence. The interested reader may also consult \cite{Mazya2008,Miclo2008,Roustant2017} for similar results.

\begin{proof}[Proof of \cref{thm:monotonicity}]
It is enough to assume $H^\prime\geq 0$ on $(0,a)\cap\mathcal{K}^c$. The other case then follows by considering $-H$.

It follows from \cref{thm:poincare_compact} and \cref{prop:sl_properties} \ref{it:properties_1} that
\begin{equation*}
	\partial_\beta\alpha_\beta=-\frac{\alpha_\beta^2}{Z_\beta}\int_{-a}^{a}\varphi_1^\beta(x)H^\prime(x)\left(\int_x^a\varphi_1^\beta(t)e^{-\beta H(t)}\,dt\right)\,dx.
\end{equation*}
We now claim that 
\begin{equation}\label{eq:claim}
	\int_{-a}^{a}f(x)H^\prime(x)\left(\int_x^af(t)e^{-\beta H(t)}\,dt\right)\,dx>0
\end{equation}
for all continuous, strictly monotone, and odd functions $f:U\to\R$. Upon establishing this claim, we can deduce the theorem thanks to \cref{prop:sl_properties} \ref{it:properties_2}.

To prove \eqref{eq:claim}, we first observe that there is no loss of generality in assuming $f$ to be strictly increasing (consider $-f$ otherwise). We further notice that, since $H$ is not constant, there is an $\varepsilon>0$ and an $x_0\in (0,a)$ such that $H^\prime(x)>0$ for all $x\in (x_0-\varepsilon,x_0+\varepsilon)\cap\mathcal{K}^c$. Finally, elementary manipulations exploiting symmetry properties of the integrands show that the left-hand side of \eqref{eq:claim} becomes
\begin{equation*}
		2\int_{0}^{a}f(x)H^\prime(x)\left(\int_x^af(t)e^{-\beta H(t)}\,dt\right)\,dx,
\end{equation*}
which is strictly positive by the observations above.
\end{proof}

\section{Application: The Poincar\'e Constant of the $O(2)$ Model}\label{sec:o_n}

Let $\Lambda$ be a finite set of nodes and let $M=(m_{i,j})_{i,j\in\Lambda}$ be a positive definite matrix with operator norm $\|M\|<c$ for some $c<2$. We further denote the unit circle in $\R^2$ by $S^1$. The $O(2)$ model is the probability measure $\nu$ with density
\begin{equation}\label{eq:o_n_model}
	\nu(d\xi)\propto e^{-\frac{1}{2}q_M(\xi)}d\xi, \qquad \xi\in(S^1)^{\Lambda},
\end{equation}
where the quadratic form $q_M$ acts as
\begin{equation*}
	q_M(\xi)=\sum_{i,j\in\Lambda}m_{i,j}\braket{\xi_i,\xi_j}.
\end{equation*}
In recent work Bauerschmidt and Bodineau proved the following result \cite{Bauerschmidt2019}:
\begin{proposition}\label{prop:bauerschmidt}
Let $M\in\R^{\Lambda\times\Lambda}$ be positive definite and assume $\|M\|<2$. Then the $O(2)$ model \eqref{eq:o_n_model} satisfies a Poincar\'e inequality uniformly with respect to the set $\Lambda$. More precisely, there is an $\eta>0$ independent of $\Lambda$ such that
\begin{equation}\label{eq:poincare_bauerschmidt}
	\var_{\nu}(f)\leq\eta\left(1+\frac{4\|M\|}{2-\|M\|}\right) \sum_{i\in\Lambda}\int |\nabla_{\xi_i} f|^2\,d\nu
\end{equation}
for all $f\in H^1_\nu\big((S^1)^\Lambda\big)$. Here, $|\nabla_{\sigma_i}f|$ denotes the length of the gradient of $f$ with respect to the $i^\textup{th}$ argument, both taken in the Riemannian sense.
\end{proposition}

\Cref{thm:monotonicity} allows us to strengthen \cref{prop:bauerschmidt} by improving the value of the Poincar\'e constant:
\begin{corollary}\label{cor:improved_poincare}
We can choose $\eta=4$ in \eqref{eq:poincare_bauerschmidt}:
\begin{equation}\label{eq:var_bound_to_prove}
	\var_{\nu}(f)\leq4\left(1+\frac{4\|M\|}{2-\|M\|}\right) \sum_{i\in\Lambda}\int |\nabla_{\xi_i} f|^2\,d\nu.
\end{equation}
\end{corollary}

\begin{proof}
The argument is for the most part similar to \cite{Bauerschmidt2019}. Nonetheless, we give the full proof for completeness. Details for some of the computations can be found in the work of Bauerschmidt and Bodineau.

Fix $c\in\big(\|M\|,2\big)$. Then we can write $M^{-1}=c^{-1}\id+B^{-1}$ for some positive definite matrix $B$. This immediately shows that
\begin{equation*}
	e^{-\frac12 q_M(\xi)}=C\int_{(\R^2)^\Lambda} e^{-\frac{c}{2}|\phi-\xi|^2} e^{-\frac12 q_B(\phi)}\,d\phi
\end{equation*}
for some numerical constant $C>0$. In particular, for any $f:(S^1)^\Lambda\to\R$, we can write
\begin{equation}\label{eq:o_n_fubini}
	\int_{(S^1)^\Lambda}f(\xi)\,\nu(d\xi)=\int_{(\R^2)^\Lambda} \left(\int_{(S^1)^\Lambda} f(\xi)\,\bar{\nu}_\phi(d\xi)\right)\,\bar{\nu}_r(d\phi),
\end{equation}
where
\begin{equation*}
	\bar{\nu}_r(d\phi)\propto e^{-\frac12 q_B(\phi)}\,d\phi,\qquad\phi\in(\R^2)^\Lambda
\end{equation*}
and $\bar{\nu}_\phi(d\xi)=\prod_{i\in\Lambda}\bar{\nu}_{\phi_i}(d\xi_i)$ for
\begin{equation*}
	\bar{\nu}_{\phi_i}(d\xi_i)\propto e^{-\frac{c}{2}|\phi_i-\xi_i|^2},\qquad\xi_i\in S^1.
\end{equation*}
For $\beta\geq 0$ define the Gibbs measure
\begin{equation}\label{eq:extended_single_spin}
	\mu_{\beta}(dx)=\frac{1}{Z_\beta} e^{\beta\cos t}dt, \quad t\in[-\pi,\pi].
\end{equation}
Note that the Poincar\'e constants of the measures $\bar{\nu}_{\phi_i}$ and $\mu_{c|\phi_i|}$ coincide. Applying \cref{thm:monotonicity} to $\mu_\beta$, we can therefore bound the Poincar\'e constant of $\bar{\nu}_{\phi_i}$ by the Poincar\'e constant of the Lebesgue measure $\mu_0$. It is well known that the latter is $\eta=4$. In particular, $\eta$ also provides an upper bound on the Poincar\'e constant of the measure $\bar{\nu}_\phi$ by the tensorization principle.

We can now prove \eqref{eq:var_bound_to_prove}. To this end, let $f\in H^1\big((S^1)^\Lambda\big)$ and abbreviate $\mu(f)\define\int f\,d\mu$. It follows from the law of total variance and \eqref{eq:o_n_fubini} that
\begin{equation}\label{eq:to_bound}
	\var_\nu(f)=\int_{(\R^2)^\Lambda}\var_{\bar{\nu}_\phi}\,\bar{\nu}_r(d\phi)+\var_{\bar{\nu}_r}\big(\bar\nu_{\phi}(f)\big).
\end{equation}
By the Poincar\'e inequality with constant $\eta$ for $\bar{\nu}_\phi$, we get
\begin{equation*}
	\int_{(\R^2)^\Lambda}\var_{\bar{\nu}_\phi}\,\bar{\nu}_r(d\phi)\leq\eta\sum_{i\in\Lambda}\int_{(S^1)^\Lambda} |\nabla_{\xi_i} f|^2\,d\nu.
\end{equation*}
To bound the second term in \eqref{eq:to_bound}, we first notice that, by \cite[Theorem D.2]{Dyson1978} and the Bakry-\'Emery criterion \cite{Bakry1983}, the measure $\bar\nu_r$ satisfies a Poincar\'e inequality with constant $\alpha=\big(c-\frac{c^2}{2}\big)^{-1}$. In particular, we get
\begin{equation}\label{eq:second_bound}
	\var_{\bar{\nu}_r}\big(\bar\nu_{\phi}(f)\big)\leq \alpha\sum_{i\in\Lambda} \bar\nu_r\big(|\partial_{\phi_i}\bar\nu_{\phi}(f)|^2\big).
\end{equation}
One can check that
\begin{equation*}
	\partial_{\phi_i}\bar\nu_{\phi}(f)=c\int_{(S^1)^\Lambda}\big(f(\xi)-\bar\nu_\phi(f)\big)\big(\xi_i-\bar\nu_\phi(\xi_i)\big)\,\bar\nu_\phi(d\xi)\define\cov_{\bar{\nu}_\phi}(f,\xi_i).
\end{equation*}
Since $\cov_{\bar{\nu}_\phi}(\cdot)=\bar{\nu}_\phi\big(\cov_{\phi_i}(\cdot)\big)$, it is enough to bound $\big|\cov_{\bar{\nu}_{\phi_i}}(f,\xi_i)\big|$. To this end, we estimate
\begin{align*}
 	\big|\cov_{\bar{\nu}_{\phi_i}}(f,\xi_i)\big|&\leq\sqrt{\var_{\bar\nu_{\phi_i}}(f)}\left(\frac12\int_{S^1\times S^1}\big|\xi_i-\tilde{\xi}_i\big|^2\,\bar{\nu}_{\phi_i}(d\xi_i)\,\bar{\nu}_{\phi_i}(d\tilde\xi_i)\right)^{\frac12}\\
 	&\leq\sqrt{2\var_{\bar\nu_{\phi_i}}(f)},
\end{align*} 
where the second inequality uses that $|\xi_i-\tilde\xi_i|\leq 2$. Applying once more the Poincar\'e inequality for $\bar\nu_{\phi_i}$, we have shown that
\begin{equation*}
	\big|\partial_{\phi_i}\bar\nu_{\phi}(f)\big|^2\leq 2c^2\bar\nu_\phi\left(\sqrt{\var_{\bar\nu_{\phi_i}}(f)}\right)^2\leq 2c^2\bar\nu_\phi\big(\var_{\bar\nu_{\phi_i}}(f)\big)\leq 2c^2\eta \bar\nu_\phi\big(|\nabla_{\xi_i} f|^2\big).
\end{equation*}
We plug this back into \eqref{eq:second_bound} and finally find from \eqref{eq:to_bound} and another application of \eqref{eq:o_n_fubini}
\begin{equation*}
	\var_\nu(f)\leq\eta\big(1+2c^2\alpha\big)\sum_{i\in\Lambda}\int_{(S^1)^\Lambda} |\nabla_{\xi_i} f|^2\,d\nu=\eta\left(1+\frac{4c}{2-c}\right)\sum_{i\in\Lambda}\int_{(S^1)^\Lambda} |\nabla_{\xi_i} f|^2\,d\nu.
\end{equation*}
Recalling that $\eta=4$, the proof is concluded by letting $c\downarrow\|M\|$.
\end{proof}
The argument given by Bauerschmidt and Bodineau \cite{Bauerschmidt2019} actually applies to the Log-Sobolev inequality, too. Recall that we say that a measure $\mu$ on $U$ satisfies a \emph{Logarithmic Sobolev inequality} with constant $\varpi\geq 0$ if
\begin{equation}\label{eq:lsi}
	\ent_{\mu}(f^2)\define\int_U f^2\log\left(\frac{f^2}{\|f\|_{L^2(U,\mu)}^2}\right)\,d\mu\leq \varpi\int_U (f^\prime)^2\,d\mu
\end{equation}
for all $f\in H^1(U,\mu)$. The quantity $\ent_{\mu}(\cdot)$ is called the \emph{entropy} of the measure $\mu$ and we agree on the convention $\ent_{\mu}(0)=0$. Note that the inequality \eqref{eq:lsi} implies in particular $\ent_{\mu}(f^2)<\infty$ for all $f\in H^1(U,\mu)$. Again, we shall denote the optimal constant in \eqref{eq:lsi} by $\varpi$ in the sequel. It turns out that the Logarithmic Sobolev inequality is stronger than the Poincar\'e inequality \eqref{eq:poincare} \cite{Deuschel1989,Rothaus1980, Rothaus1981}.

The main obstacle in transferring \cref{cor:improved_poincare} to the Log-Sobolev constant is the fact that \cref{thm:monotonicity} cannot hold in the same generality, see \cref{ex:monotonicity_lsi} below. Of course, we do not really need a result as strong as \cref{thm:monotonicity} in order to derive \cref{cor:improved_poincare} for the Log-Sobolev inequality but only wish to show monotonicity of the Log-Sobolev constant of the single-spin measure $\mu_\beta$ in \eqref{eq:extended_single_spin}. Unfortunately, this currently seems to be out of reach.

\begin{example}\label{ex:monotonicity_lsi}
Let $U=(-1,1)$. Given a cutoff $\gamma\in(0,1)$, we consider the Hamiltonian
\begin{equation*}
	H_\gamma(x)= |x|\wedge\gamma.
\end{equation*}
Then the associated Gibbs measure clearly falls in the regime of \cref{thm:monotonicity} and the Poincar\'e constant is thus strictly monotone decreasing in the inverse temperature $\beta$. Defining the functions
\begin{align*}
	f_\beta(x)&\define\int_{x}^{1}e^{-\beta H_\gamma(t)}\,dt=e^{-\beta\gamma}(1-(x\vee\gamma))+\frac{e^{-\beta(x\wedge\gamma)}-e^{-\beta\gamma}}{\beta},\\
	g_\beta(x)&\define\int_{0}^{x}e^{\beta H_\gamma(t)}\,dt=e^{\beta\gamma}(x-(x\wedge\gamma))+\frac{e^{\beta(x\wedge\gamma)}-1}{\beta}
\end{align*}
for $x\geq 0$, we see that the quantities \eqref{eq:lower_lsi} and \eqref{eq:upper_lsi} below become
\begin{align*}
	b_\beta&=\sup_{0\leq x\leq 1}f_\beta(x)\log\left(1+\frac{f_\beta(0)}{f_\beta(x)}\right)g_\beta(x),\\
	B_\beta&=\sup_{0\leq x\leq 1}f_\beta(x)\log\left(1+\frac{2e^2f_\beta(0)}{f_\beta(x)}\right)g_\beta(x).
\end{align*}
Evaluating these two expressions numerically, we find the plot depicted in \cref{fig:bounds_lsi}. Thus appealing to \cref{prop:muckenhoupt_lsi}, we can conclude that the Log-Sobolev constant of this measure is not monotone. 
\begin{figure}[H]
	\centering
	\includegraphics[width=0.8\textwidth]{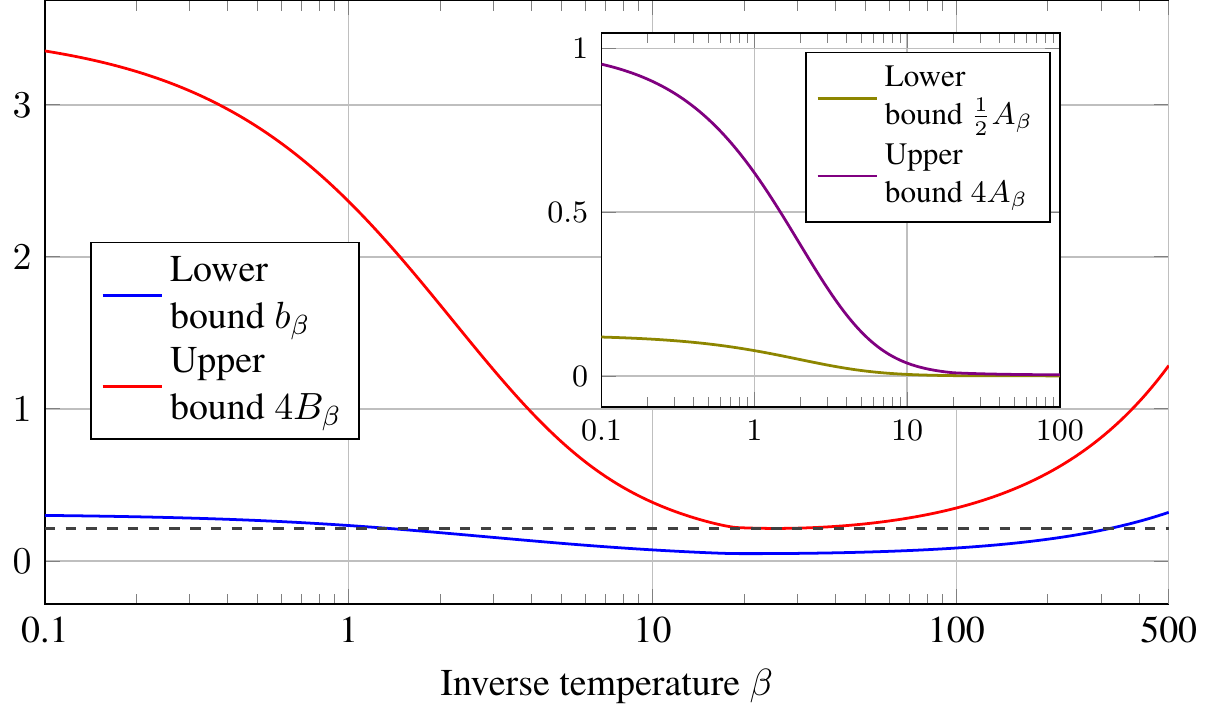}
	\caption{Plot of the lower and upper bounds on the Log-Sobolev constant of the measure $d\mu_\beta^\gamma(x)\propto e^{-\beta (|x|\wedge\gamma)}dx$ for $\gamma=.95$. As indicated by the dashed horizontal line, the constant can not be monotone in $\beta$. The inset shows the respective bounds for the Poincar\'e constant provided by \cref{prop:muckenhoupt_poincare}. Notice that the $x$-axes admit logarithmic scaling.}
	\label{fig:bounds_lsi}
\end{figure}
\end{example}

\section{Saturation of the Log-Sobolev Constant}\label{sec:saturation}

In this final section we show that the measure $\mu_\beta$ defined in \eqref{eq:extended_single_spin} exhibits a noteworthy property: The Log-Sobolev constant saturates at a strictly positive value in the zero-temperature limit whereas the Poincar\'e constant vanishes:
\begin{theorem}\label{thm:saturation}
	Let $\alpha_\beta$ and $\varpi_\beta$ denote the the Poincar\'e and Log-Sobolev constants of the Gibbs measure 
	\begin{equation*}
		\mu_\beta(dx)\propto e^{\beta\cos x}\1_{[-\pi,\pi]}(x)\,dx.
	\end{equation*}
	 Then there are constants $0<c_1<c_2$ independent of $\beta$ such that
	\begin{equation}\label{eq:lsi_bounds}
	c_1\leq \varpi_\beta\leq c_2
	\end{equation}
	for all $\beta\geq 0$. Similarly, there exists a uniform constant $c_3>0$ such that
	\begin{equation*}
	\alpha_\beta\leq \frac{c_3}{\beta}
	\end{equation*}
	for all $\beta> 0$.
\end{theorem}
We give the proof of \cref{thm:saturation} in \cref{sec:proof_saturation}. It relies on the bounds presented in the following section.

\subsection{Muckenhoupt's Bounds}

The following bounds on the Poincar\'e constant follow from the work of Muckenhoupt \cite{Muckenhoupt1972}, see e.g. \cite{Ane2000,Bobkov1999}. The version we state below features improved numerical factors and was obtained by Miclo \cite{Miclo2008}, see also \cite{Mazya2008}.
\begin{proposition}\label{prop:muckenhoupt_poincare}
	Let $U=[-a,a]$, $a>0$, be a symmetric interval and assume that the function $H:U\to\R$ in the definition of the Gibbs measure $\mu_\beta$ is even. Then the sharp constant $\alpha_\beta$ in \eqref{eq:poincare} satisfies
	\begin{equation*}
		A_\beta\leq \alpha_\beta\leq 4 A_\beta,
	\end{equation*}
	where 
	\begin{equation*}
		A_\beta=\sup_{0\leq x\leq a}\left(\int_{x}^{a}e^{-\beta H(t)}\,dt\right)\left(\int_0^x e^{\beta H(t)}\,dt\right).
	\end{equation*}
\end{proposition}
There is an analogous result for the Log-Sobolev constant, which is originally due to Bobkov and G\"otze \cite{Bobkov1999}. We state the criterion in a sharpened version by Barthe and Roberto \cite{Barthe2003}.
\begin{proposition}\label{prop:muckenhoupt_lsi}
	Let $U=[-a,a]$, $a>0$, be a symmetric interval and assume that the function $H:U\to\R$ in the definition of the Gibbs measure $\mu_\beta$ is even. Then the sharp constant $\varpi_\beta$ in \eqref{eq:poincare} satisfies
	\begin{equation*}
    b_\beta\leq \varpi_\beta\leq 4B_\beta,
	\end{equation*}
	where 
	\begin{align}
	b_\beta&=\sup_{0\leq x\leq a}\left(\int_x^ae^{-\beta H(t)}\,dt\right)\log\left(1+\frac{Z_\beta}{2\int_x^ae^{-\beta H(t)}\,dt}\right)\left( \int_0^x e^{\beta H(t)}\,dt\right),\label{eq:lower_lsi}\\
	B_\beta&=\sup_{0\leq x\leq a}\left(\int_x^ae^{-\beta H(t)}\,dt\right)\log\left(1+\frac{e^2 Z_\beta}{\int_x^ae^{-\beta H(t)}\,dt}\right)\left( \int_0^x e^{\beta H(t)}\,dt\right)\label{eq:upper_lsi}
	\end{align}
	with the convention $0\cdot\infty=\infty\cdot 0=0$. Furthermore, we have $B_\beta\leq 4b_\beta$ and therefore $b_\beta\leq \varpi_\beta\leq 16 b_\beta$.
\end{proposition}

\subsection{Proof of \cref{thm:saturation}}\label{sec:proof_saturation}

This section is devoted to the proof of \cref{thm:saturation}. Let us first derive estimates on the partition function $Z_\beta\define\int_{-\pi}^\pi e^{\beta\cos t}\,dt$ of the Gibbs measure \eqref{eq:extended_single_spin}: 
\begin{lemma}\label{lem:partition_function}
	We have that
	\begin{equation*}
		\pi e^{\frac{2\beta}{\pi}}\leq Z_\beta\leq 2\pi e^{\beta}
	\end{equation*}
	for all $\beta\geq 0$. Moreover, $\beta\mapsto Z_\beta$ is increasing.
\end{lemma}
\begin{proof}
	The upper bound is immediate. For the lower bound we make use of Jensen's inequality to estimate
	\begin{equation*}
	 	\int_{-\pi}^\pi e^{\beta\cos t}\,dt\geq 2\int_{0}^{\pi/2}e^{\beta\cos t}\,dt\geq\pi\exp\left(\frac{2\beta}{\pi}\int_0^{\pi/2}\cos t\,dt\right)=\pi e^{\frac{2\beta}{\pi}}.
	 \end{equation*}
	 To see that $\beta\mapsto Z_\beta$ is increasing, we take the derivative
	 \begin{equation*}
	 	\frac{d}{d\beta}Z_\beta=2\int_0^\pi \cos (t) e^{\beta\cos t}\,dt=2\left(\int_{0}^{\pi/2}\cos (t) e^{\beta\cos t}\,dt-\int_{0}^{\pi/2}\cos (t) e^{-\beta\cos t}\,dt\right)\geq 0
	 \end{equation*}
	 since $e^{\beta\cos t}\geq e^{-\beta\cos t}$ for any $t\in [0,\pi/2]$.
\end{proof}

Next, we give a technical lemma from which we then, in combination with \cref{prop:muckenhoupt_poincare,prop:muckenhoupt_lsi}, deduce \cref{thm:saturation} effortlessly.
\begin{lemma}\label{lem:technical_results}
Let $\beta>0$. Then for $x\in [0,\pi]$ the following bounds hold:
\begin{enumerate}
	\item\label{it:technical_1} We have the lower bounds
	\begin{align}
	\int_{x}^{\pi}e^{\beta\cos t}\,dt &\geq \frac{1}{\sqrt{2\beta}e^{\beta}}\int_{0}^{\beta(1+\cos x)}\frac{e^t}{\sqrt{t}}\,dt,\label{eq:lower_bound_1}\\
	\int_{0}^{x}e^{-\beta\cos t}\,dt&\geq \frac{e^\beta}{\sqrt{2\beta}}\int_{\beta(1+\cos x)}^{2\beta}\frac{e^{-t}}{\sqrt{t}}\,dt.\label{eq:lower_bound_2}
	\end{align}
	\item\label{it:technical_2} Furthermore, the following explicit bounds hold:
\begin{align}
	\sqrt{2}e^{-\beta}\sqrt{1+\cos x}\leq\int_{x}^{\pi}e^{\beta\cos t}\,dt&\leq \frac{4e^{\beta\cos x}}{\sqrt{\beta}},\label{eq:bound_1}\\
	\int_{0}^{x}e^{-\beta\cos t}\,dt&\leq \frac{4}{e^{\beta\cos x}\sqrt{\beta}}.\label{eq:bound_2}
\end{align}
If $\beta>1$, the upper bound of \eqref{eq:bound_1} can be tightened as follows:
\begin{equation}\label{eq:bound_3}
	\int_x^\pi e^{\beta\cos t}\,dt\leq\frac{4e^{\beta\cos x}}{\sqrt{\beta}}\left(1\wedge\sqrt{\beta(1+\cos x)}\right).
\end{equation}
\end{enumerate}
\end{lemma}

We provide a proof of \cref{lem:technical_results} at the end of this section.

\begin{proof}[Proof of \cref{thm:saturation}] The theorem is proven by deriving suitable bounds on the quantities $A_\beta$ and $b_\beta$ from \cref{prop:muckenhoupt_poincare,prop:muckenhoupt_lsi} respectively.
	
We shall begin with the lower bound on $b_\beta$ and thus on the Log-Sobolev constant $\varpi_\beta$. Since we are interested in a lower bound on the supremum, we can certainly assume that $x\in [\pi/2,\pi)$. We treat the cases $\beta\leq 1$ and $\beta>1$ separately.

\dashuline{$\beta>1$:}
Invoking \cref{lem:partition_function} and \eqref{eq:bound_1}, we obtain
\begin{align*}
	\log\left(1+\frac{Z_\beta}{2\int_x^\pi e^{\beta\cos t}\,dt}\right)&\geq\log\left(1+\frac{\pi\sqrt{\beta} e^{\frac{2\beta}{\pi}}}{8 e^{\beta\cos x}}\right)\geq \log\left(1+\frac{\pi e^{\frac{2\beta}{\pi}}}{8}\right)\\
	&\geq \log\left(1+\frac{\pi e^{\frac{\pi\beta}{8}}}{8}\right)\geq \frac{\beta}{2+\frac{\pi}{4}}
\end{align*}
uniformly in $x\in [\pi/2,\pi)$. Here, we used that $\beta>1$, $\cos x\leq 0$, and the elementary inequality $\log\big(1+a e^{\frac{x}{a}}\big)\geq \frac{x}{2(1+a)}$ for $a,x>0$. To see the latter, note that the inequality holds at $x=0$ and check that $\frac{1}{2(1+a)}\leq\frac{e^{\frac{x}{a}}}{1+ae^{\frac{x}{a}}}=\frac{d}{dx}\log\big(1+a e^{\frac{x}{a}}\big)$ for all $a,x>0$.

To see that $b_\beta\geq c_1>0$ uniformly in $\beta>1$, it is enough to show that
\begin{equation*}
	\inf_{\beta>1}\frac{1}{\beta}\inf_{x\in[\pi/2,\pi)}\left(\int_x^\pi e^{\beta\cos t}\,dt\right)\left(\int_0^x e^{-\beta\cos t}\,dt\right)>0.
\end{equation*}
In view of \eqref{eq:lower_bound_1} and \eqref{eq:lower_bound_2}, this follows upon proving that $\inf_{\beta>1}\eta_\beta>0$ where
\begin{equation*}
	\eta_\beta\define\sup_{\pi/2\leq x\leq\pi}\left(\int_0^{\beta(1+\cos x)}\frac{e^t}{\sqrt{t}}\,dt\right)\left(\int_{\beta(1+\cos x)}^{2\beta}\frac{e^{-t}}{\sqrt{t}}\,dt\right).
\end{equation*}
But this in turn easily follows from the observation
\begin{align*}
	\eta_\beta&=\sup_{0\leq x\leq1}\left(\int_0^{\beta x}\frac{e^t}{\sqrt{t}}\,dt\right)\left(\int_{\beta x}^{2\beta}\frac{e^{-t}}{\sqrt{t}}\,dt\right)\\
	&=4\sup_{0\leq x\leq 1}\left(\int_0^{\sqrt{\beta x}}e^{t^2}\,dt\right)\left(\int_{\sqrt{\beta x}}^{\sqrt{2\beta}}e^{-t^2}\,dt\right)\nonumber\\
	&\geq 4\left(\int_{0}^{1/\sqrt{2}}e^{t^2}\,dt\right)\left(\int_{1/\sqrt{2}}^{\sqrt{2}}e^{-t^2}\,dt\right)\simeq .814157.
\end{align*}
In summary, we found $b_\beta\geq\frac{\eta_\beta}{4+\frac{\pi}{2}}\define d_1>0$ uniformly in $\beta> 1$.

\dashuline{$\beta\leq 1$:} Here, we just take $x=\pi/2$ to find
\begin{equation*}
\frac{\pi}{2e}\leq\int_{\pi/2}^{\pi}e^{\beta\cos t}\,dt=\int_{0}^{\pi/2}e^{-\beta\cos t}\,dt\leq \frac{\pi}{2}.
\end{equation*}
Consequently, we obtain
\begin{equation*}
	b_\beta\geq \left(\frac{\pi}{2e}\right)^2\log\left(1+\frac{Z_\beta}{2\int_{\pi/2}^\pi e^{\beta\cos t}\,dt}\right)\geq \left(\frac{\pi}{2e}\right)^2\log(3)\define d_2>0
\end{equation*}
where we used that $Z_\beta\geq Z_0=2\pi$ for all $\beta\geq 0$, see \cref{lem:partition_function}.

Setting $c_1\define d_1\wedge d_2$, we thus deduced $\varpi_\beta\geq c_1$ uniformly in $\beta\geq 0$.

Let us now turn to the upper bound on $\varpi_\beta$. Again, we distinguish the cases $\beta> 1$ and $\beta\leq1$.

\dashuline{$\beta\leq 1$:} The function $[0,\infty)\times [0,\infty)\ni (x,\zeta)\mapsto x\log(1+\zeta/x)$ is strictly increasing in both arguments. It follows
\begin{equation*}
b_\beta\leq \left(\int_{0}^{\pi}e^{\cos t}\,dt\right)^2\log\left(1+\frac{Z_1}{2\int_{0}^{\pi}e^{\cos t}\,dt}\right)\define\overline{d}_1.
\end{equation*}

\dashuline{$\beta>1$:} Using \cref{lem:partition_function} and the lower bound from \eqref{eq:bound_1}, we obtain
\begin{equation*}
\log\left(1+\frac{Z_\beta}{2\int_x^\pi e^{\beta\cos t}\,dt}\right)\leq\log\left(1+\frac{\pi e^{2\beta}}{\sqrt{2(1+\cos x)}}\right)\leq 2\beta +\log\left(1+\frac{\pi}{\sqrt{2(1+\cos x)}}\right)
\end{equation*}
for $x\in [0,\pi)$. Invoking the upper bounds \eqref{eq:bound_2} and \eqref{eq:bound_3} one has
\begin{align*}
b_\beta&\leq 32+\frac{16}{\beta}\sqrt{1+\cos x} \log\left(1+\frac{\pi}{\sqrt{2(1+\cos x)}}\right)\\
&\leq 32+16\sqrt{2}\log\left(1+\frac{\pi}{2}\right)\define \overline{d}_2.
\end{align*}
We choose $c_2\define 16(\overline{d}_1\vee\overline{d}_2)$ to establish the upper bound on the Logarithmic Sobolev constant. 

This finishes the proof of the first part of the theorem. Similar arguments establish the asserted upper bound on the Poincar\'e constant. We leave the details to the reader.
\end{proof}

We conclude the article with the elementary integral bounds employed in the preceding proof.
\begin{proof}[Proof of \cref{lem:technical_results}]
We first note that the substitution $y=\cos t$ gives
\begin{equation}\label{eq:tech_sim}
	\int_x^\pi e^{\beta\cos t}\,dt=\int_{-1}^{\cos x}\frac{e^{\beta y}}{\sqrt{1-y^2}}\,dy.
\end{equation}
The denominator can be bounded by $\sqrt{1-y^2}\leq\sqrt{2(1+y)}$, $y\in[-1,1]$. We use this and the substitution $z=\beta(1+y)$ to find
\begin{equation}\label{eq:x_to_pi_lower}
\int_x^\pi e^{\beta\cos t}\,dt\geq \frac{1}{\sqrt{2\beta}e^{\beta}}\int_{0}^{\beta(1+\cos x)}\frac{e^{z}}{\sqrt{z}}\,dz
\end{equation}
for all $x\in [0,\pi]$. This is the lower bound \eqref{eq:lower_bound_1}. A similar computation establishes \eqref{eq:lower_bound_2}.

Let us now turn to part \ref{it:technical_2} of the lemma. We first observe the elementary bound
\begin{equation}\label{eq:int_bounds_exp/sqrt}
	2\sqrt{r}\leq\int_0^r\frac{e^z}{\sqrt{z}}\,dz\leq 2e^r(1\wedge\sqrt{r})
\end{equation}
for all $r\geq 0$. Applying this to \eqref{eq:x_to_pi_lower} immediately yields the lower bound in \eqref{eq:bound_1}. For the upper bound, we have to distinguish two cases. Let us first assume $x\in [\pi/2, \pi]$. Then using $\sqrt{1-y^2}\geq \sqrt{1-|y|}$, $y\in[-1,1]$, for the denominator in \eqref{eq:tech_sim} followed by the substitution $z=\beta(1+y)$, we arrive at
\begin{equation}\label{eq:x_to_pi_1}
	\int_x^\pi e^{\beta\cos t}\,dt\leq\frac{1}{\sqrt{\beta}e^\beta}\int_0^{\beta(1+\cos x)}\frac{e^z}{\sqrt{z}}\,dz\leq\frac{2e^{\beta\cos x}}{\sqrt{\beta}}\left(1\wedge\sqrt{\beta(1+\cos x)}\right)
\end{equation}
where the last step used \eqref{eq:int_bounds_exp/sqrt}. This establishes the upper bounds in \eqref{eq:bound_1} and \eqref{eq:bound_3} in this case.

If $x\in[0,\pi/2)$, we first follow the same chain of manipulations as before---but now with the substitution $z=\beta(1-y)$---to find
\begin{equation*}
	\int_x^{\pi/2}e^{\beta\cos t}\,dt\leq \frac{e^\beta}{\sqrt{\beta}}\int_{\beta(1-\cos x)}^{\beta}\frac{e^{-z}}{\sqrt{z}}\,dz\leq\frac{2e^{\beta\cos x}}{\sqrt{\beta}}
\end{equation*}
where we used $\int_r^\infty e^{-z}/\sqrt{z}\,dz\leq 2e^{-r}$ for all $r\geq 0$. Combining this with \eqref{eq:x_to_pi_1}, we can deduce the upper bound in \eqref{eq:bound_1}:
\begin{equation*}
	\int_x^\pi e^{\beta\cos t}\,dt\leq\frac{2}{\sqrt{\beta}}+\frac{2e^{\beta\cos x}}{\sqrt{\beta}}\leq \frac{4e^{\beta\cos x}}{\sqrt{\beta}}.
\end{equation*}
If, in addition, $\beta> 1$, then certainly $(1\wedge\sqrt{\beta(1+\cos x)})=1$ and \eqref{eq:bound_3} follows.

For \eqref{eq:bound_2}, we compute analogously
\begin{align*}
	\int_0^xe^{-\beta\cos t}\,dt&\leq \frac{2}{\sqrt{\beta}e^{\beta\cos x}},\quad x\in[0,\pi/2],\\
	\int_{\pi/2}^xe^{-\beta\cos t}\,dt&\leq \frac{2}{\sqrt{\beta}e^{\beta\cos x}},\quad x\in(\pi/2,\pi],
\end{align*}
from which the asserted upper bound easily follows.
\end{proof}

\footnotesize
\bibliographystyle{abbrv}
\bibliography{monotonicity_v2.bib}
\end{document}